\documentclass[11pt,a4paper]{article}
\usepackage[utf8]{inputenc}
\usepackage[english]{babel}

\usepackage{amsmath}
\usepackage{amsthm}
\usepackage{amsfonts}
\usepackage{amssymb}
\usepackage[colorlinks=false,pdfborderstyle={/W 1}]{hyperref}

\def\arXiv#1#2{\href{http://front.math.ucdavis.edu/#1}{{\tt arXiv:#1 [#2]}}}

\usepackage{tikz}
\usepackage{dsfont}
\usepackage{bbm}
\usetikzlibrary{arrows}
\tikzset{
    >=stealth,
    auto,
    node distance=2cm,
    font=\tiny,
    state/.style={circle,draw,minimum size=30pt ,thick,align=center}
}

\usepackage{graphicx}

\newtheorem{theo}{Theorem}[section]

\newtheorem{lemma}[theo]{Lemma}
\newtheorem{propo}[theo]{Proposition}

\newtheorem{conj}[theo]{Conjecture}

\numberwithin{equation}{section}
\numberwithin{figure}{section}

\newcommand{\zar}[1]{\left(#1\right)}

\newcommand{\norm}[1]{\left\lVert#1\right\rVert}

\def\mix{\mathrm{mix}}
\def\IP{\mathrm{IP}}
\def\Binom{\mathsf{Binom}}
\def\bgiv{\,\big|\,}
\def\Bgiv{\,\Big|\,}
\def\p{\partial}
\def\cG{\mathcal{G}}
\def\cB{\mathcal{B}}
\def\cE{\mathcal{E}}
\def\lora{\longrightarrow}

\def\K{\mathcal{K}}

\def\ind#1{\mathds{1}\lbrace#1\rbrace}
\def\indic#1#2{f_{#1}^{#2}}
\def\1{\mathds{1}}

\def\Pss#1#2{\P_{#1}\zar{#2}}
\def\Ess#1#2{\E_{#1}\zar{#2}}

\def\Gnm{\hat{G}_{n,m}}
\def\R{\mathbb{R}}
\def\Z{\mathbb{Z}}
\def\eps{\varepsilon}

\def\hatpr#1{\hat{X}_{#1}}
\def\tsproc#1{\widetilde{X}_{#1}}
\def\O{\mathcal{O}}

\def\E{\mathbb{E}}
\def\P{\mathbb{P}}
\def\Var{\mbox{\rm Var}}

\def\md{\mid}
\def\Bb#1#2{{\def\md{\bigm| }#1\bigl(#2\bigr)}}
\def\BB#1#2{{\def\md{\Bigm| }#1\Bigl(#2\Bigr)}}
\def\Bs#1#2{{\def\md{\mid}#1(#2)}}
\def\Pb{\Bb\P}
\def\Eb{\Bb\E}
\def\PB{\BB\P}
\def\EB{\BB\E}
\def\Ps{\Bs\P}
\def\Es{\Bs\E}
\def\Vars{\Bs\Var}
\def\Varb{\Bb\Var}
\def\VarB{\BB\Var}

\def\eref#1{(\ref{#1})}
\def \proof {{ \medbreak \noindent {\bf Proof.} }}
\def\proofof#1{{ \medbreak \noindent {\bf Proof of #1.} }}

\def\qqed{\qed\medskip}

\title{Mixing time and cutoff phenomenon\\ for the interchange process on dumbbell graphs\\
and the labelled exclusion process on the complete graph}

\author{Rich\'ard Patk\'o and G\'abor Pete \footnote{
'R\'enyi Institute, Not The Hungarian Academy of Sciences, and Institute of Mathematics, Budapest University of Technology and Economics. \url{http://www.math.bme.hu/\~rpatko} and  \url{http://www.math.bme.hu/\~gabor}.}
}

\date{}

\begin{document}
\maketitle

\abstract{We find the total variation mixing time of the interchange process on the dumbbell graph (two complete graphs, $K_n$ and $K_m$, connected by a single edge), and show that this sequence of chains exhibits the cutoff phenomenon precisely when the smaller size $m$ goes to infinity. The mixing time undergoes a phase transition at $m\asymp \sqrt{n}$. We also state a conjecture on when exactly cutoff holds for the interchange process on general graphs. 

Our proofs use coupling methods, and they also give the mixing time of the simple exclusion process of $k$ labelled particles in the complete graph $K_n$, for any $k\leq n$, with cutoff, as conjectured by Lacoin and Leblond (2011). In particular, this is a new probabilistic proof for the mixing time of random transpositions, first established by Diaconis and Shahshahani (1981).}

\section{Introduction}

Given a connected finite graph $G\zar{V,E}$ on $N$ vertices, the discrete time {\bf lazy interchange process} is a group-invariant random walk on the symmetric group $S_N$, where each step of the walk consists of either staying put with probability one half, or multiplying the element $\sigma_t$ at time $t$ with a transposition $(v_1v_2)$, where $(v_1,v_2)$ is an edge of $G$ selected uniformly at random.

The mixing time of an interchange process was first studied by Diaconis and Shahshahani \cite{diaconis}, who proved, using the representation theory of the symmetric group, that the {\bf total variation mixing time} of the lazy process on the complete graph $K_n$ is $({1}+o(1)) \, n\log n$, with {\bf cutoff phenomenon}: the total variation distance of the distribution of the current location from stationarity decreases from around 1 to around 0 abruptly, in a time window negligible compared to the mixing time itself (see Subsection~\ref{ss.def} below for precise definitions). Probabilistic proofs were given in \cite{Matt,BSZ,BeSe}. For general graphs, a central result is the proof of Aldous' conjecture by Caputo, Liggett, and Richthammer \cite{liggett}: the spectral gap of the process is always determined by the spectral gap of simple random walk on the underlying graph itself. However, understanding the mixing time requires more than just finding the spectral gap; results on different graph families have been obtained, in chronological order, by Jonasson \cite{jonasson}, Erikshed \cite{erik}, Oliveira \cite{oliveira}, Lacoin \cite{lacoin,lacoinProfile}, Hermon and Pymar \cite{HerPym}, Alon and Kozma \cite{octopi}, Hermon and Salez \cite{HerSal}.

The question of finding conditions that ensure or forbid the cutoff phenomenon was posed by Aldous and Diaconis \cite{cutoff}, and has been studied in many papers since then; see \cite{DLP,BHP} and the references therein. In any sequence of Markov chains, if the product of the spectral gap and the mixing time does not tend to infinity (i.e., the so-called {\bf product condition} fails), then  cutoff cannot hold, as shown by the eigenvector corresponding to the second largest eigenvalue \cite[Section 18.3]{LPW}. One may speculate that if a sequence of ``natural'' Markov chains does not have this obvious obstacle, then cutoff does hold. For transitive Markov chains, there is an example due to Pak that satisfies the product condition while has no cutoff \cite[ibid.]{LPW}, but the case of interchange processes is wide open. From a different point of view, Cayley graphs of finite simple groups tend to be {\bf expanders} \cite{helfgott}, where the product condition obviously holds, and cutoff is conjectured by Peres \cite{LubPer}. Hence it is not immediately obvious if there are any graphs where the interchange process (which is a random walk on a Cayley graph of $S_N$, a $\Z_2$-extension of the simple group $A_N$) does not satisfy the product condition.

The {\bf dumbbell graphs} are typical examples with bad mixing properties for the simple random walk, hence it is natural to investigate what happens for the interchange process on them. To state our first result, we let $t_{\mathrm{mix}}\zar{\varepsilon}$ denote, as usual, the smallest time when the total variation distance of the chain from stationarity, started from a fixed vertex, gets below $\eps$. 

\begin{theo}[{\bf Dumbbell interchange}]\label{maintheo}
Let $G=G_{n,m}$ be the graph consisting of two complete graphs with $n$ and $m$ vertices respectively, connected by a single edge. Let us assume that $m=m(n)$ is a function of $n$ such that $m\leq n$. Note that the edges set has size $|E|\sim (m^2+n^2)/2$. Then we have the following for the $\frac12$-lazy interchange process on this graph:
\begin{enumerate}
\item[i)] If $\exists\:c>0$, such that $c\sqrt{n}\leq m \leq n$, then,  for all $0<\eps<1/2$,
\begin{equation}\label{e.largem}
(1-o(1))\frac{|E|nm}{n+m}\log {n}\leq t_{\mathrm{mix}}(\eps)\leq(1+o(1))\frac{|E|nm}{n+m}\log {n}
\end{equation}
\item[ii)] If we have $1\ll m=m(n)\ll\sqrt{n}$, then,  for all $0<\eps<1/2$,
\begin{equation}\label{e.smallm}
(2-o(1))|E|m\log {m}\leq t_{\mathrm{mix}}(\eps)\leq(2+o(1))|E|m\log {m}
\end{equation}
\item[iii)] If $m$ remains bounded, then, for all large enough $n$ and all $0<\eps<1/2$,
\begin{equation}\label{e.constm}
A(\eps,m)\, |E| \leq t_{\mathrm{mix}}(\eps) \leq B(\eps,m)\, |E|,
\end{equation}
where $A(\eps_m,m) > B(1/4,m)$ if $\eps_m$ is small enough.
\end{enumerate}
In particular, the interchange process has cutoff if and only if $m(n)\rightarrow\infty$, which is exactly when the product condition holds.
\end{theo}

Inspired by these results, we state the following somewhat provocative conjecture. It is a formulation of the idea that the interchange process can fail to have cutoff only if some local obstacle governs the mixing time. This is not a purely graph theoretical characterization that would be immediate to check, but it still may be a good start.

\begin{conj}[{\bf Interchange cutoff}]\label{mainconj} Let $G_n=(V_n,E_n)$ be a sequence of finite simple graphs, with $t_\mix^\IP(G_n)$ denoting the total variation mixing time of the discrete time interchange process.
\begin{itemize}
\item[{\bf (i)}] Cutoff holds if and only if the product condition holds.
\item[{\bf (ii)}] If  $t_\mix^\IP(G_n)\gg |E_n|$, then cutoff holds.
\item[{\bf (iii)}] If $t_\mix^\IP(G_n) \leq O(|E_n|)$, then cutoff fails if and only if the graphs $G_n$ have {\bf bounded bad bottlenecks}: there exists $K<\infty$ such that, for all $n$ large enough, 
\begin{equation}\label{e.bbb}
 \exists \; W_n \subset V_n \text{ with }1 \leq |W_n| \leq K \text{ and } |\p_E W_n | \leq \frac{K\, |E_n|}{t_\mix^\IP(G_n)}\,, 
\end{equation}
where $\p_E S$ is the set of edges connecting $S$ with its complement.
\end{itemize}
\end{conj}

One direction is easy: we will show in Proposition~\ref{p.bbb} that the existence of the bounded bad bottlenecks $W_n$ implies that the product condition fails. Also, one simply cannot have $W_n\subset V_n$ with $|W_n|\to \infty$ and $|\partial_E W_n | = {O( |E_n|)}/{t_\mix^\IP(G_n)}$. Let us mention that our conjecture seems to be closely related to the natural conjecture that Hermon and Pymar's \cite[Theorem 1.4]{HerPym} holds with an exponent $1/2$ instead of $1/4$; furthermore, it is consistent with the conjectures of \cite{oliveira} and \cite{HerPym} comparing the mixing time of the interchange process with the mixing time of independent particles.
\medskip

In order to better place Theorem~\ref{maintheo} in context, and to state our second theorem, we need to say a few words about the proof of the first one. The first step is to replace the ``bridge'' edge connecting the two cliques with $nm$ ``thin bridges'', one edge for each pair of vertices in different cliques, with probability $\frac{1}{2|E|nm}$ of choosing the transposition represented by the edge. That is, we ``distribute the probability'' between the new edges equally, and hence the new process can be coupled to the original process so that one of the new edges is chosen exactly when the old edge is chosen. A key observation will be that these ``bridge'' transpositions happen so rarely that the {\it permutation of the particles within in each clique} typically has enough time to get mixed in between them. (We emphasize that this is only the typical behaviour, and there are still many instances of bridge transpositions occurring shortly after each other, causing non-trivial complications; see Proposition~\ref{p.coupling1}. Moreover, when $m$ is constant, we symmetrize only the connections to the $n$-clique; see Proposition~\ref{p.coupling2}. Nevertheless, to explain the big picture, let us stick to this simplistic view.) This has two consequences: 
\begin{itemize}
\item[{\bf (1)}] The original and the new process can be coupled so that mixing in the two happen simultaneously. Thus we can study the mixing time of the new (much more symmetric) process.
\item[{\bf (2)}] The mixing time of the new process is determined by the mixing time of how the particles are partitioned into the two cliques. Moreover, because of the symmetries of the new process, this partition  is always uniformly distributed among all possibilities with a given number of particles that are not in the clique where they started. Thus we only have to understand the mixing time of the {\it number of particles} that started in the smaller clique and are currently there. This process is just a time-changed version of the {\bf Bernoulli--Laplace diffusion model}, with urn sizes $n$ and $m$. In that model, at each step, one ball is chosen from each urn uniformly at random, then the two balls are switched. In our case, we make moves only when bridge transpositions happen.
\end{itemize}
 
Given this reduction, we need the mixing time of the Bernoulli--Laplace model, which has been determined by different methods in earlier works. Diaconis and Shahshahani  \cite{bernoulli}  use the representation theory of the symmetric group both for the lower and upper bounds. Their proof is spelled out for the $m=n$ case, establishing cutoff, and they mention that everything goes through for the case of general $m$, overlooking the phase transition (and the need for a different argument) for $m = O(\sqrt{n})$. An elementary algebraic approach was recently given in \cite{LimPick}. More probabilistically, after noticing that the  Bernoulli--Laplace model can be considered as a birth-and-death chain on $\{0,1,\dots,m\}$, the work of Ding, Lubetzky and Peres \cite{DLP} can be applied, who prove that cutoff for these chains is equivalent with the product condition, with mixing time given by the expected hitting time of the median of the stationary distribution, starting from the worse endpoint. However, even though one can write a recursion for expected hitting times, solving the recursion explicitly is not a completely trivial task analytically (done heuristically in \cite{LimPick}); moreover, for the case $m = o(\sqrt{n})$ it cannot give the precise answer, since the median is ``inside'' the $0$ state. Finally, Lacoin and Leblond \cite{lale} give a complete probabilistic treatment of the {\bf simple exclusion process} on the complete graph $K_n$, with $k=k(n)$ labelled or unlabelled particles, which means that the location of these particles only is followed during the interchange process. In the unlabelled case, by considering the set of original particle locations as one urn, the empty locations as the other urn, we get the Bernoulli--Laplace model. In that paper, cutoff for $k\gg 1$ and the phase transition at $k \asymp \sqrt{n}$ were established. The labelled case has some direct similarities with the interchange process on $\Gnm$; however, the exact mixing time was not found in \cite{lale} for the labelled process. The $L^2$-mixing time of the labelled process with $k\leq (1-\eps)n/2$, different from the total variation mixing time, was found in \cite{FoJo} using spectral arguments.

Our proofs for the Bernoulli--Laplace model are in parts similar to those of Lacoin and Leblond \cite{lale}. However, since we learnt about that paper only after our first draft was written, and we feel that our proofs are simpler at several places, we have decided to present these proofs in detail, keeping our paper self-contained. Moreover, we are able to complete their work on the labelled exclusion process, and prove the following theorem. After seeing our draft, Hubert Lacoin suggested that the strong stationary time method of Matthews \cite{Matt} might also generalize to the case of $k$ particles. This suggestion indeed seems to work, which would give a completely different proof.

\begin{theo}[{\bf Complete graph exclusion}]\label{t.lalenew}
For the $\frac12$-lazy exclusion process on the complete graph with $n$ vertices and $k$ labeled particles, for any $1\leq k\leq n$,
for every $\eps\in (0,1)$, we have
$$
t_\mix(\eps) = \, n\log k + O_\eps(n)\,,
$$
where the constant in the error term depends on $\eps$ but not on $k$.
\end{theo}

We are working here with the $\frac12$-lazy processes only for convenience. Even without laziness, the exclusion process would be aperiodic for any $1\leq k\leq n-2$, and the result could be proved for the non-lazy version with mixing time halved. However, for $k=n-1$ or $n$, the process is periodic (there are odd and even permutations), hence something needs to be done about that. In the usual version for $k=n$, there is a natural $\frac1n$-laziness, which is large enough to produce mixing between even and odd permutations, but small enough to keep the mixing time at $\frac{1+o(1)}{2}n\log n$. To keep our results consistent for all values of $k$, we have decided to go with  $\frac12$-laziness.
\medskip
 
\noindent {\bf Sketch for the Bernoulli--Laplace model.} When $m \ll \sqrt{n}$, with high probability no original particle can be found in the smaller clique in the stationary distribution, and it is not hard to prove that the phenomenon of all particles leaving the smaller clique actually governs mixing. For $1 \ll m \ll \sqrt{n}$ one can rather directly compute the time needed for that, obtaining cutoff. For $m \asymp 1$, we have to work not with the Bernoulli--Laplace model, but with a ``half-symmetrized chain'', for which direct probabilistic arguments easily give the mixing time, with no cutoff. In the case $m=\Omega\zar{\sqrt{n}}$, mixing turns out to be governed by the number of original particles in the smaller clique being as close to the stationary mean as the stationary standard deviation. The lower bound is found by Chebyshev's inequality on the number of original particles residing in the larger clique. For this reason, the mean and variance of this quantity at time $t$, are calculated using the eigendecompositions of transition matrices of projections of the original chain. The upper bound is found by coupling of two copies of the Bernoulli--Laplace process. We also note that the mixing time for the $m\asymp\sqrt{n}$ case can be obtained by plugging in $m\asymp\sqrt{n}$ in either one of the two cases. 
\medskip

\noindent {\bf Sketch for the exclusion process.} We consider the set of positions where the particles start as one clique $\K$, and the starting empty locations as another clique $\K^c$, to get a process very similar to the previous ones. The particles leaving $\K$ are arriving at uniform random locations, hence the mixing time is governed by two phenomena: (i) the number of particles that are in $\K$ should be close to stationarity, which is basically a Bernoulli--Laplace process; (ii) the particles that have never left $\K$ should be well-mixed within $\K$, which is basically another exclusion process. The time $(1+o(1)) \, n\log k$ is just enough for item (i). On the other hand, the number of particles that have never left $\K$ during this time turns out to be $o(\sqrt{k})$, which takes us into the easier case of the exclusion process, and the number of transpositions happening within $\K$ is again just enough for item (ii). Of course, there are some complications coming from the fact that the ``particles that have never left $\K$'' are somewhat special, but this effect will turn out to be unimportant.
\medskip

\noindent {\bf Organization of paper.} In Section~\ref{s.prelim}, we present some basic definitions and results, including a proof of the easy direction of Conjecture~\ref{mainconj}. In Section~\ref{s.redu}, we present the reduction to the Bernoulli--Laplace model and what we call the half-symmetrized chain, including the proofs of items (1) and (2) above. In Section~\ref{s.symm}, we present the computations of the mixing time of the (half-)symmetrized chains, completing the proof of Theorem~\ref{maintheo}. Finally, in Section~\ref{s.exclu}, we prove Theorem~\ref{t.lalenew} on the mixing time of the exclusion processes.
\medskip

\noindent {\bf Acknowledgments.} We are grateful to Bal\'azs R\'ath, Jonathan Hermon and Hubert Lacoin for useful discussions, comments and references. Our work was supported by the ERC Consolidator Grant 772466 ``NOISE'', and by the Hungarian National Research, Development and Innovation Office, NKFIH grant K109684.

\section{Preliminaries}\label{s.prelim}

\subsection{Mixing time definitions}\label{ss.def}

If $\mu$ and $\nu$ are probability measures on the finite set $\Omega$, then the {\bf total variation distance}, denoted by $\norm{\mu-\nu}$, is defined as
$$\norm{\mu-\nu}:=\max_{A\subset \Omega}|\mu\zar{A}-\nu\zar{A}|.$$

Let $X^{(n)}$ be a sequence of irreducible aperiodic Markov chains on the (finite) state spaces $\Omega^{(n)}$. Let $P_n^t(\cdot,\cdot)$ be the $t$-step transition matrix, and let $\pi_n$ be the stationary distribution of $X^{(n)}$. Let 
$$d_n(t):=\max_{x\in \Omega^{(n)}}\norm{P_n^t(x,\cdot)-\pi_n}.$$ 
Then, for $0<\varepsilon< 1$, the $\varepsilon$-mixing time of $X^{(n)}$ is defined as $t_{\mathrm{mix}}^{(n)}\zar{\varepsilon}=\inf\lbrace t>0\mid d_n(t)<\varepsilon \rbrace$. By {\bf cutoff phenomenon} for the sequence $X^{(n)}$ we mean that, for any $0<\varepsilon< 1$,
$$
\frac{t_{\mathrm{mix}}^{(n)}(\varepsilon)}{t_{\mathrm{mix}}^{n}(1-\varepsilon)}\rightarrow 1,\:\:\:\text{as}\:\:n\rightarrow\infty.
$$

Our Markov chains will be reversible and lazy, hence the eigenvalues of $P=P_n$ are $0\leq \lambda_n \leq \dots \leq \lambda_1=1$. For the {\bf spectral gap}, the Dirichlet variational formula \cite[Lemma 13.7]{LPW} says that
\begin{equation}\label{e.gapDir}
1-\lambda_2 = \inf \left\{\frac{\cE(f)}{ \|f\|_2^2 } : \sum_{x} f(x)\pi(x)=0 \right\},
\end{equation} 
where 
$$
\cE(f):=\frac{1}{2} \sum_{x,y\in\Omega} \big(f(x)-f(y)\big)^2 \pi(x) P(x,y)
\qquad\text{and}\qquad
\|f\|_2^2:=\sum_{x\in\Omega} f(x)^2 \pi(x)\,.
$$
The relaxation time is $t_{\mathrm{relax}}^{(n)}:=1/(1-\lambda_2^{(n)})$, and it is easy to see that the {\bf product condition}
\begin{equation}\label{e.prodcond}
t_{\mathrm{relax}}^{(n)} \ll t_{\mathrm{mix}}^{(n)}\,,
\quad\text{i.e.,}\quad
(1-\lambda_2^{(n)})\,  t_{\mathrm{mix}}^{(n)} \to \infty
\end{equation}
must be satisfied in order for the cutoff phenomenon to hold \cite[Section 18.3]{LPW}.

\subsection{Basics of the stationary distribution and the symmetrized chain}\label{ss.basic}

In the dumbbell graph $G_{n,m}$, let $\K_1$ and $\K_2$ be the two cliques, with vertex sets identified with $\lbrace 1,\ldots, n\rbrace$ and $\lbrace n+1,\ldots, n+m\rbrace$, respectively, with a single edge connecting $n$ and $n+1$. The stationary distribution of the interchange process is of course uniform on the symmetric group $S_N$, with $N=n+m$. As a corollary, the stationary distribution of the number of particles from $\lbrace n+1,\ldots, n+m\rbrace$ that reside in $\K_1$ is the {\bf hypergeometric distribution} $\mathrm{HypGeom}(N,m,n)$: from a population size $N$, with $m$ marked individuals, in a sample of size $n$, this is the random number of marked individuals. It is well-known and not hard to prove that 
\begin{equation}\label{e.HypGeom}
\begin{aligned}
\EB{ \mathrm{HypGeom}(n+m,m,n) } &= \frac{m n}{m+n}\,,\\
\Var \Big( \mathrm{HypGeom}(n+m,m,n) \Big) &= \frac{m^2 n^2}{(m+n)^2(m+n-1)} .
\end{aligned}
\end{equation}

%

As mentioned in the Introduction, we will consider a {\bf symmetrized underlying graph} $\Gnm$: we replace the ``bridge'' edge $(n,n+1)$ with $nm$ ``thin bridges'', one edge $(i,j)$ for each $i\in\K_1$ and $j\in\K_2$ with ``weight'' $1/(nm)$, meaning that, in the interchange process, the probability of choosing one of these edges is $1/(2|E|nm)$ instead of the usual ${1}/{(2|E|)}$. 

The virtue of this symmetrization is that now the projection of the interchange process that follows the movement of a single particle between the cliques is still Markovian, shown on Figure~\ref{f.M2}.

\begin{figure}[h!]
\centering
\begin{tikzpicture}
\node[state] (1) {$\K_1$};
\node[state] (2) [right of=1] {$\K_2$};
\draw[-latex] (1) to[bend left=30] node[above] {$\frac{1}{2|E|n}$} (2);
\draw[-latex] (2) to[bend left=30] node[below] {$\frac{1}{2|E|m}$} (1);
\path[]
    (1) edge [loop left] node {$1-\frac{1}{2|E|n}$} (1)
    (2) edge [loop right] node {$1-\frac{1}{2|E|m}$} (2);
\end{tikzpicture}
\caption{Following a single particle between the cliques in the interchange process over $\Gnm$.}\label{f.M2}
\end{figure}
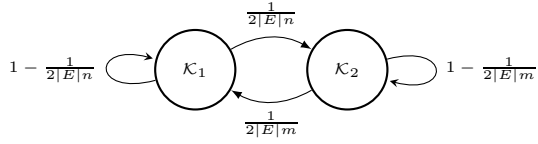

The eigenvalues of the transition matrix of this Markov chain are $1$ and $1-\frac{n+m}{2|E|nm} $, with right eigenvectors 
$(1,\, 1)$ and $\zar{-\frac{m}{n},\, 1}$, which form an orthonormal basis w.r.t.~the stationary distribution $\zar{\frac{n}{n+m},\,\frac{m}{n+m}}$. The spectral gap $\frac{n+m}{2|E|nm}$ of this projection is obviously an upper bound on the spectral gap of the original chain, hence we get a lower bound $c_\eps \frac{|E|nm}{n+m}$ on the $\eps$-mixing time; see \cite[Theorem 12.5]{LPW}. A better lower bound is given by {\bf Wilson's method} \cite{Wilson}, \cite[Theorem 13.28]{LPW}: if we denote the second eigenvalue by $\lambda_2$, and the second eigenvector, as a function on the state space $(\K_1,\K_2)$ by $\phi_2$, then we can lift it to an eigenvector $\Phi$ of the interchange process on $\Gnm$, with the same eigenvalue, by 
$$
\Phi(\sigma):=\sum_{i=n+1}^{n+m}\phi_2(\sigma_i)\,;
$$ in words, we look at where $\sigma$ took the particles of $\K_2$ and count $-\frac{m}{n}$ for all that have been moved to $\K_1$ and $1$ for all that have not been moved out of $\K_2$. 
One can easily show that 
$$
R:=\max_{\sigma\in S_{n+m}}\E_{X_0=\sigma} \big|\Phi(X_1)-\Phi(\sigma)\big|^2 = \frac{1}{2|E|}\zar{1+\frac{m}{n}}^2,
$$
and then Wilson's bound is
\begin{equation}\label{wilson}
\begin{aligned}
t_{\text{mix}}^{n,m}\geq  \frac{1}{2\log\zar{\frac{1}{\lambda_2}}}\log\zar{\frac{\Phi(\mathrm{id})^2(1-\lambda_2)}{2R}}\sim \frac{|E|mn}{m+n}\log\zar{\frac{mn}{m+n}}.
\end{aligned}
\end{equation}
Our results will show that this bound is useful only in the extreme cases: when $m\asymp 1$, it demonstrates a lack of cutoff, while in the case of $m \asymp n$ it is actually sharp. 

\subsection{Bounded bad bottlenecks}\label{ss.bbb}

We will now prove the following proposition regarding the easy direction of Conjecture~\ref{mainconj}:

\begin{propo}\label{p.bbb}
Let $G_n=(V_n,E_n)$ be a sequence of finite simple graphs, and let $t_\mix^\IP(G_n)$ denote the total variation mixing time of the discrete time interchange process over $G_n$.
\begin{itemize}
\item[{\bf (i)}] If the bounded bad bottlenecks~(\ref{e.bbb}) exist, then the product condition~(\ref{e.prodcond}) does not hold.
\item[{\bf (ii)}] There cannot exists $W_n\subset V_n$ with $|W_n|\to \infty$, $|W_n|<|V_n|/2$, and $|\partial_E W_n | = {O( |E_n|)}/{t_\mix^\IP(G_n)}$.
\end{itemize}
\end{propo}

\begin{proof} {\bf (i)} Assume the existence of a sequence $W_n$ satisfying~(\ref{e.bbb}). For any vertex $v_n \in W_n$, we have $\deg(v_n)-|W_n| \leq |\partial_E W_n|$, hence there exists some constant $\tilde K<\infty$ such that $\deg(v_n) \leq \tilde K |E_n| / t_\mix^\IP(G_n)$. Fix such a vertex $v=v_n$, and consider the function $f: S_N \lora \R$ defined by $f(\sigma)=\mathbf{1}_{\{\sigma(v)=v\}}$. We will use this $f$ to give an upper bound on the spectral gap via~(\ref{e.gapDir}).

Firstly, $\|f\|^2_2=1/N$. In order to compute $\cE(f)$, notice that $f(\sigma_1)\not=f(\sigma_2)$ for a pair of permutations with $P(\sigma_1,\sigma_2)>0$ if and only if $\sigma_1^{-1}\sigma_2$ is a transposition given by an edge emanating from $v$, and one of the $\sigma_i$'s fixes $v$. Therefore, $\cE(f)=\frac{\deg(v)}{|E_n|}\frac{1}{N}$. Altogether,~(\ref{e.gapDir}) gives that
$$
t_{\mathrm{relax}}^\IP(G_n) \geq \frac{|E_n|}{\deg(v)} \geq \frac{t_\mix^\IP(G_n)}{\tilde K}\,,
$$
and hence the product condition~(\ref{e.prodcond}) indeed fails.

{\bf (ii)} Assume that there does exist such a sequence of subsets $W_n$. In $t_{\mathrm{mix}}(G_n)$ steps, the expected number of transpositions using edges in $\p_E W_n$ is $t_\mix^\IP(G_n) |\p_E W_n|/|E_n|=O(1)$. Therefore, the number of particles in the complement $W_n^c$ at time $t_\mix^\IP(G_n)$ that started in $W_n$ remains tight.

On the other hand, in the stationary distribution, the number of particles in $W_n^c$ that started in $W_n$ has a $\mathrm{HypGeom}(N,|W_n|,N-|W_n|)$ distribution, with mean $(N-|W_n|)|W_n|/N \geq |W_n|/2 \gg 1$ and standard deviation $\sim (N-|W_n|)|W_n|/N^{3/2}$ (see~(\ref{e.HypGeom})). Thus, by Chebyshev's inequality, this random number goes to infinity in probability. Comparing with the previous tightness, we get that the distribution at time $t_{\mathrm{mix}}(G_n)$ has total variation distance close to 1 from stationarity, contradicting the definition of mixing time. 
\qqed
\end{proof}

\section{Reduction to the Bernoulli-Laplace chain}\label{s.redu}

We will always think of the $\frac12$-laziness of the interchange process as first attempting a transposition, then actually doing it only with probability $1/2$. Hence, by time $t$, there are $t$ {\bf attempted transpositions}.

\subsection{Mixing within the cliques}\label{ss.within}

As mentioned in the Introduction, the discrete time interchange process on the complete graph $K_k$, with $1/2$ laziness, has total variation mixing time $t_\mix^\IP(K_k)=\zar{1+o(1)} k \log k$ \cite{diaconis,BSZ,BeSe}. A well-known general fact \cite[Eq.~(4.33)]{LPW}, valid in any Markov chain, is that after the mixing time, the total variation distance from stationarity is decaying exponentially fast on the scale of the mixing time. In our case, there exists a universal constant $c>0$ such that, for all $L>L_0=1$,
\begin{equation}\label{e.expfast}
d_{K_k}^\IP(L\,k\log k) < \exp(-c L)\,.
\end{equation}
We will not  use the sharp result that $L_0=1$ works here; it is much easier to prove the result for $L_0=4$ \cite[Corollary 8.10]{LPW}, and this will actually suffice for us. (This will be important when we say that our arguments give a new proof of $t_\mix^\IP(K_k)=\zar{1+o(1)} k \log k$ itself.)

We will apply the bound~(\ref{e.expfast}) to the interchange process over $G_{n,m}$ and $\Gnm$ restricted to the cliques $\K_1$ and $\K_2$. By this restriction we mean that we simply ignore the effect of the bridge transpositions (they are considered as lazy no-moves).

\begin{propo}\label{quickmix}
Consider the lazy interchange process over $G_{n,m}$ or $\Gnm$, restricted to the cliques, as defined above. Fix $\eps>0$, and let $t=\frac{n^{2}}{m^{1-\eps}}$. Then, at time $t$, the total variation distance of the process from having independent uniform permutations in both cliques is at most $\tilde C \exp\zar{ -\tilde c\, \frac{m^{\eps}}{\log m} }$, with some absolute constants $0<\tilde c,\tilde C < \infty$.

When $m$ is a constant, then, at time $t=n^{1+\eps}$, the total variation distance of the permutation in $\K_1$ from the uniform distribution is at most $\tilde C \exp\zar{ -\tilde c\, \frac{n^{\eps}}{\log n} }$, with some constants $0<\tilde c,\tilde C < \infty$ that may depend only on $m$.
\end{propo}

\proof
The number of attempted transpositions  within $\K_1$ has a binomial distribution: $T_1 \stackrel{d}{=} \Binom\zar{t,\zar{1-\frac{1}{|E_n|}} \frac{\binom{n}{2}}{\binom{n}{2} + \binom{m}{2}} }$. 
This has expectation at least $c_1 n^{1+\eps}$, so, by a standard large deviations estimate, the probability that it is less than $\frac{c_1}{2} n^{1+\eps}$ is at most $\exp(-\tilde c_1 n^{1+\eps})$. If the number of these attempted transpositions is indeed at least $\frac{c_1}{2} n^{1+\eps}$, then, by~(\ref{e.expfast}), the distribution of the resulting permutation within $\K_1$ has total variation distance less than $\exp(-\frac{c\, c_1}{2} \,n^{\eps} / \log n)$ from uniform.

The number of attempted transpositions within $\K_2$ is $T_2 \stackrel{d}{=} \Binom\zar{t,\zar{1-\frac{1}{|E_n|}} \frac{\binom{m}{2}}{\binom{n}{2} + \binom{m}{2}} }$. This has expectation at least $c_2 m^{1+\eps}$, and is less than $\frac{c_2}{2} m^{1+\eps}$ with probability $<\exp(-\tilde c_2 m^{1+\eps})$. If it is not this small, then, by~(\ref{e.expfast}), the distribution of the resulting permutation within $\K_2$ has total variation distance less than $\exp(-\frac{c\, c_2}{2} \,m^{\eps} / \log m)$ from uniform.

The attempted transpositions within $\K_1$ and $\K_2$ are independent from each other, except for their numbers ($T_1+T_2$ plus the number of attempted bridge transpositions equals $t$). Thus, if $X_i(t)$ denotes the configuration within $\K_i$ at time $t$, then
\begin{equation}\label{e.product1}
\begin{aligned}
\Pb{X_1(t)=\sigma_1,\, X_2(t)=\sigma_2} 
&=
\EB{\Pb{X_1(t)=\sigma_1,\, X_2(t)=\sigma_2 \md T_1, T_2}}\\
&= \EB{\Pb{X_1(t)=\sigma_1 \md T_1} \, \Pb{X_2(t)=\sigma_2 \md T_2} }\,. 
\end{aligned}
\end{equation}
Let $\pi_i$ denote the stationary distribution of $X_i(t)$, and for notational ease, let us write $f_i(\sigma_i):=\Ps{X_i(t)=\sigma_i \md T_i}$. Then, the total variation distance, conditionally on $(T_1,T_2)$, can be written as
\begin{equation}\label{e.product2}
\begin{aligned}
\sum_{\sigma_1,\sigma_2} \Big| \Ps{X_1(t)=\sigma_1,\, X_2(t)=\sigma_2 \md T_1,T_2} -\pi_1(\sigma_1) \, \pi_2(\sigma_2) \Big| & \\
&\hskip - 2 in =
\sum_{\sigma_1,\sigma_2} \big| f_1(\sigma_1) \, f_2(\sigma_2) - \pi_1(\sigma_1) \, \pi_2(\sigma_2) \big| \\
&\hskip - 2 in \leq
\sum_{\sigma_1,\sigma_2}  \Big\{ \big| f_1(\sigma_1) - \pi_1(\sigma_1) \big| f_2(\sigma_2) + \pi_1(\sigma_1)  \big| f_2(\sigma_2) - \pi_2(\sigma_2) \big| \Big\} \\
&\hskip - 2 in = \sum_{\sigma_1} \big| f_1(\sigma_1) - \pi_1(\sigma_1) \big| +   \sum_{\sigma_2} \big| f_2(\sigma_2) - \pi_2(\sigma_2) \big| \,.
\end{aligned}
\end{equation}
The first term is the TV-distance of $\pi_1$ and $X_1(t)$, conditioned on $T_1$. With probability at least $1-\exp(-\tilde c_1 n^{1+\eps})$, this $T_1$ is such that this TV-distance is at most $\exp(-\frac{c\, c_1}{2} \,n^{\eps} / \log n)$. So, the expectation  over $T_1,T_2$ of the first term is at most $\exp(-\tilde c_1 n^{1+\eps})+\exp(-\frac{c\, c_1}{2} \,n^{\eps} / \log n)$. Similarly, the expectation over $T_1,T_2$ of the second term is at most $\exp(-\tilde c_2 m^{1+\eps}) + \exp(-\frac{c\, c_2}{2} \,m^{\eps} / \log m)$. Altogether,  the total variation distance of the chain from the product of the uniform distributions is at most $\tilde C \exp\zar{ -\tilde c\, \frac{m^{\eps}}{\log m} }$, as desired.

When $m$ is a constant, then the number of transpositions in $\K_1$ by time $t=n^{1+\eps}$ is at least $c_1 n^{1+\eps}$ with probability at least $1-\exp(-\tilde c_1 n^{1+\eps})$, and if this event happens, then the total variation distance from uniform is at most $\exp(- {c\, c_1} n^{\eps} / \log n)$ by \eref{e.expfast}, hence the claim follows.
\qqed

In the interchange process over $\Gnm$, every time a bridge transposition happens, uniform random particles get moved, and they arrive at uniform random places. Moreover, the attempted bridge transpositions are independent of the chain restricted to the cliques, except for the number of steps in the two chains. Using the argument of (\ref{e.product1},\ref{e.product2}), if enough time has passed so that, with high probability, the number of steps within the cliques is beyond the mixing time given by Proposition~\ref{quickmix}, and also the number of bridge transpositions is large enough so that the \emph{number of particles in a given clique} that started from that clique has mixed,  then the full system has mixed. As noted in the Introduction, the latter process is a time-changed Bernoulli--Laplace diffusion model with two urns, one containing $n$ and the other containing $m$ balls.  The mixing time for this process will turn out to be of larger order than the mixing time within the cliques, hence that will be the dominant term.

\subsection{Coupling between the original and the symmetrized process}

We will prove later that $t^\IP_\mix(\Gnm)$ satisfies the bounds (\ref{e.largem}, \ref{e.smallm}) of Theorem~\ref{maintheo}. Note that both cases satisfy $t^\IP_\mix(\Gnm) \asymp n^2 \, m \log m$, which is much larger, for $m(n)\to\infty$, than the time scale $\frac{n^2}{m^{1-\eps}}$ of Proposition~\ref{quickmix}. This makes it possible to prove the following statement.

\begin{propo}\label{p.coupling1}
Assume $m(n)\rightarrow \infty$, and assume that we already know (proved later) that $t^\IP_\mix(\Gnm) \asymp n^2 \, m \log m$, with cutoff. Then $t^\IP_\mix(G_{n,m}) \sim t^\IP_\mix(\Gnm)$, also with cutoff. 
\end{propo}

\proof We will write whp for ``with high probability'', i.e., for a probability tending to 1. Let $X_t$ be the interchange process on $G_{n,m}$, and $\hatpr{t}$ the interchange process on $\Gnm$. We are going to couple $X_t$ to a third process, $\tsproc{t}$, which will just be a time-changed version of $\hatpr{t}$ with a small time shift, while $\Ps{X_t \neq \tsproc{t}}$ will be small for all relevant values of $t$.

Whenever we choose an edge within a clique in $G_{n,m}$, then we can choose the same edge in $\Gnm$ too, since the distribution of choices is the same. Failure of the coupling can only occur if we choose ``the bridge'' in $G_{n,m}$. In this case we have to choose what happens in $\Gnm$, where there are $nm$ tiny bridges. Notice that the time between bridge transpositions is a geometric random variable with mean $\asymp n^2$, hence the probability of choosing a bridge before $\frac{n^2}{m^{1-\eps}}$ is at most $m^{-1+\eps}$. We will call this event, with any $\eps\in(0,1/2)$ fixed, a {\bf short run}, and the complement a {\bf long run}. By Proposition~\ref{quickmix}, after a long run, the order of the particles, conditioned on the identity of the particles in the cliques, is very close to being uniform. Hence the permutation given by a long run, followed by a bridge transposition, can be coupled to be the same in $X_t$ and $\hatpr{t}$, with a small probability $\tilde C \exp\zar{ -\tilde c\, \frac{m^{\eps}}{\log m} }$ of failure. 

What makes life more difficult is that, during the order $n^2\,m\log m$ steps, there are order $m\log m$ bridge transpositions (with probability at least $1-\exp(-c m \log m)$), hence short runs do happen. However, the probability of two short runs right after each other is of order $m^{-2+2\eps}$, hence whp this is not going to happen during our order $m\log m$  runs, so we will be able to ignore this possibility.

We will present three cases of a long run followed by a short run, which is then followed by another long run. In the process $X_t$, in the first long run, whp there is a uniform mixing of particles in both cliques (a permutation denoted by $U_1$). Then we have a bridge transposition $T_1$, then a non-uniform mixing $V_2$ in the cliques by the short run, then another bridge transposition $T_2$, and a final permutation $U_3$ in the cliques that is again close to uniform. From these permutations, we will produce a time-shifted process $\tsproc{t}$ on $\Gnm$.

In the first case, assume that in $G_{n,m}$, the particles on the bridge remain fixed by the permutation $V_2$. Then $T_2$ simply switches back $T_1$, and the final permutation $U_3$ reshuffles the cliques uniformly, as if $U_1,T_1,V_2,T_2$ had never happened. So, the part  $U_1,T_1,V_2,T_2,U_3$ of the process $X_t$ will be coupled to a single $\tilde U_3$ in $\tsproc{t}$. 

In the second case, exactly one of the particles leaves the bridge in $G_{n,m}$ under $V_2$. Assume that after $U_1$, we had particles $(x,y)$ on the bridge (left and right side, respectively), and that the left one leaves under $V_2$. After $V_2$, we have $(x',x)$ on the bridge, where $x'\neq x$ from $\K_1$. After $T_2$, we have $(x,x')$, then the system reshuffles under $U_3$, with some pair $(x'',y'')$ on the bridge at the end. Note that the distribution of $x'$ is uniform among the particles present in $\K_1$ before $U_1$ (it is uniform among particles different from $x$, but $x$ is uniform itself), and $y$ is uniform among the particles of $\K_2$. Thus, the effect of $U_1,T_1,V_2,T_2,U_3$ can be imitated by a sequence $\tilde U_1,\tilde T_1,\tilde U_3$ in $\tsproc{t}$. The case of the right side particle leaving the bridge can be handled similarly.

In the third case, both particles leave the bridge in $G_{n,m}$ under $V_2$. This means that $T_1$ and $T_2$ happen to uniform random elements, except that $T_2$ can choose neither particle from $T_1$. In $\Gnm$, we can simulate this by just two independent uniformly random transpositions. Assume that $(x,y)$ were the particles on the bridge in $G_{n,m}$ before $T_1$. Then the probability that $\Gnm$ chooses $x$ or $y$ to imitate $T_2$ is of order $\frac{1}{m}$. Thus, the effect of $U_1,T_1,V_2,T_2,U_3$ can be imitated by a sequence $\tilde U_1,\tilde T_1,\tilde V_2,\tilde T_2,\tilde U_3$ in $\tsproc{t}$ with probability $1-O\zar{\frac{1}{m}}$.

If there are no two short runs right after each other, then we can do the coupling between $X_t$ and $\tsproc{t}$ going through all the short runs one-by-one from the beginning (possibly using the uniform permutation $U_3$ after a short run as the uniform distribution $U_1$ preceding the next short run).

Define $\O_t$ to be the event that for all time $s\leq t$, the coupling ``is OK'': there are no two consecutive short runs, at every long run the permutations have been mixed sufficiently, and in every short run the ``bad part'' of the third case discussed above, of probability $O(1/m)$, did not occur. What is the probability of $\O_t$, when $t \asymp n^2\,m\log m$? We already know that whp the number of runs is of order $m\log m$, and similarly, the number of short runs is of order $m^{\eps}\log {m}$ (with a failure probability that is exponentially small in $m^\eps$). Let us condition on these events. Then, the probability that there is a long run where the permutations did not mix sufficiently is $O(m\log m \exp(-\tilde c\, m^\eps/\log m))$, which tends to 0. The probability that there are two short runs after each other is $O(m^{-1+2\eps}\log {m})$, which tends to 0 if $\eps< 1/2$. The probability that the bad part of the third case occurs during any of the short runs is at most  $O(m^{\eps-1}\log {m})$. So, altogether, $\Ps{\O_t} \ge 1-O(m^{-1+2\eps}\log {m})$.

Conditioned on $\O_t$, the permutations $X_t$ and $\tsproc{t}$ are the same. Conditioned on the complement $\O_t^c$, their total variation distance  is at most 1. Thus, for any $t \asymp n^2\,m\log m$,
\begin{equation}\label{tsdist}
d_{\text{TV}}\zar{X_t,\tsproc{t}} \leq \Ps{\O_t^c}  =  O(m^{-1+2\eps}\log {m})\,.
\end{equation}

We now give a bound on the order of the time we time-shifted to get $\tsproc{t}$. For each short run, the time shift is at most (from the first case above) the total length of a long run, a short run, and two transpositions. For the $O(m^\eps \log m)$ short runs, the total length is $O(n^2 m^\eps \log m)$ whp. Note that this is of smaller order than the mixing time $n^2 \, m \log m$ for $\Gnm$.

Let us denote $\hat t := t^\IP_\mix(\Gnm)$, and let $\delta\in (0,1)$ arbitrary. By the triangle inequality:
\begin{equation*}
d_{\text{TV}}\zar{X_{(1+\delta)\hat t},\pi} \leq d_{\text{TV}}\zar{X_{(1+\delta)\hat t},\tsproc{(1+\delta)\hat t}}+d_{\text{TV}}\zar{\tsproc{(1+\delta)\hat t},\pi}\,.
\end{equation*}

The first term in the sum is $o(1)$ by~\eref{tsdist}. The second term is $o(1)$ because the time shift from $\tsproc{(1+\delta)\hat t}$ to the process $\hatpr{t}$ is smaller than $\delta \hat t / 2$ whp, hence we can use the smallness of $d_{\text{TV}}\zar{\hatpr{(1+\delta/2)\hat t},\pi}$.

For a lower bound, we use the following:
\begin{equation}\label{e.lower}
\begin{aligned}
d_{\text{TV}}\zar{X_{(1-\delta)\hat t},\pi}&\geq d_{\text{TV}}\zar{X_{(1-\delta)\hat t},\pi \bgiv \O_{(1-\delta)\hat t}} \Ps{\O_{(1-\delta)\hat t}}\\
&=d_{\text{TV}}\zar{\tsproc{(1-\delta)\hat t},\pi \bgiv \O_{(1-\delta)\hat t}} \Ps{\O_{(1-\delta)\hat t}}\,.
\end{aligned}
\end{equation}
To estimate the last expression, notice that
\begin{equation*}
\begin{aligned}
d_{\text{TV}}\zar{\tsproc{(1-\delta)\hat t},\pi}&=d_{\text{TV}}\zar{\tsproc{(1-\delta)\hat t},\pi \bgiv \O_{(1-\delta)\hat t}}\Ps{\O_{(1-\delta)\hat t}}\\
&\qquad+d_{\text{TV}}\zar{\tsproc{(1-\delta)\hat t},\pi \bgiv \O_{(1-\delta)\hat t}^c}\Ps{\O_{(1-\delta)\hat t}^c}\,,
\end{aligned}
\end{equation*}
therefore
\begin{equation*}
\begin{aligned}
d_{\text{TV}}\zar{\tsproc{(1-\delta)\hat t},\pi \bgiv \O_{(1-\delta)\hat t}} {\Ps{\O_{(1-\delta)\hat t}}}
&\geq {d_{\text{TV}}\zar{\tsproc{(1-\delta)\hat t},\pi}-\Ps{\O_{(1-\delta)\hat t}^c}}\\
&\geq {d_{\text{TV}}\zar{\hat{X}_{(1-\delta)\hat t},\pi}-\Ps{\O_{(1-\delta)\hat t}^c}}\\
&={1-o(1)-o(1)}\,,
\end{aligned}
\end{equation*}
where the inequality in the second line used that $\tsproc{t}$ is just a slower version of $\hatpr{t}$. This shows that~\eref{e.lower} is $1-o(1)$, finishing the proof of Proposition~\ref{p.coupling1}.
\qqed

In the case when $m$ remains a constant, the above coupling would not work. So, we will use a ``half-symmetrized'' graph $G'_{n,m}$ instead of $\Gnm$: we replace the bridge edge of $G_{n,m}$ between $n$ and $n+1$ by $n$ small bridges $\big\{(i,n+1),\  i=1,\dots,n\big\}$, each with weight $1/n$.

\begin{propo}\label{p.coupling2}
For $m$ constant, $\eps\in (0,1/2)$ fixed, the interchange process $X_t$ on $G_{n,m}$ can be coupled to the interchange process $X'_t$ on $G'_{n,m}$ such that $\Ps{X_t = X'_t \textrm{ for all } t\leq n^{2+\eps}} = 1-o(1)$ as $n\to\infty$. 
\end{propo}

\proof In time $n^{2+\eps}$, the number of bridge transpositions is of order $n^\eps$ whp, and the probability that any of the runs between them has length less than $n^{1+\eps}$ is at most $O(n^\eps n^{1+\eps} / n^2 )$, which tends to 0. Condition on having order $n^\eps$ runs, and on all of them being at least of length $n^{1+\eps}$. By the second part of Proposition~\ref{quickmix}, at the end of each run, the permutation in $\K_1$ can be considered to be uniform whp, so can be coupled to the process $X'_t$. The coupling fails with conditional probability $O(n^\eps \exp(-\tilde c\, n^\eps / \log n))=o(1)$. Altogether, the coupling fails with probability $O(n^{2\eps-1})=o(1)$, and we are done.
\qqed

Given Propositions~\ref{p.coupling1} and~\ref{p.coupling2}, it is now enough to find $t^\IP_\mix(\Gnm)$ for $m(n)\to\infty$, and $t^\IP_\mix(G'_{n,m})$ for $m\asymp 1$, and the statements of Theorem~\ref{maintheo} will follow.

\section{Mixing in the symmetrized chains}\label{s.symm}

Following the paragraph after Proposition~\ref{quickmix}, we will focus on how the number of particles that started in $\K_1$ and are presently in $\K_1$ evolves.

\subsection{The lower bound in the case $c\sqrt{n}\leq m \leq n$}\label{ss.large1}

Let us now assume that $m=\Omega\zar{\sqrt{n}}$, but $m\leq n$. In this case we can construct a lower bound for the total variation mixing time in the following way. Let $p$ be a particle and 
$$
\indic{p}{t}=\ind{p\in \K_2 \:\text{at time}\: t};\qquad L^t:=\sum_{p\:\text{started in}\: \K_1}\indic{p}{t}\,.
$$ 

We need to find a time $t$, as large as possible, for which the number $L^t$ of particles that started in $\K_1$ and are now in $\K_2$ is still different from the typical number $L^\infty$ in the stationary distribution. We want to apply Chebyshev's inequality, hence we need the expectation and the variance of $L^t$. For these, we have
\begin{equation}\label{vareqmain}
\begin{aligned}
\Eb{L^t} &= \sum_{p\:\text{started in}\: \K_1} \Eb{\indic{p}{t}},\\
\Varb{L^t}
&=\sum_{\substack{p\neq q\\p,q\:\text{started in}\:\K_1}}\mathrm{Cov}\zar{\indic{p}{t},\indic{q}{t}}+\sum_{p\:\text{started in}\:\K_1}\VarB{\indic{p}{t}}.
\end{aligned}
\end{equation}
The expectation and variance of $\indic{p}{t}$ can be calculated using the eigenvalues and eigenvectors of the single-particle chain of Figure~\ref{f.M2}: we write $\pi(x)P^t(x,y)=(\1_x,P^t\1_y)_\pi$, where $\pi$ is the stationary distribution of the chain, then decompose $\1_x$ and $\1_y$ in the basis of eigenvectors, and apply $P^t$, to get:
\begin{equation*}
\begin{aligned}
\Eb{\indic{p}{t}}&=\PB{p\in\K_2\:\text{at time}\:t\Bigm| p\:\text{started in}\:\K_1}
=\frac{m}{m+n}-\frac{m}{m+n}\zar{1-\frac{m+n}{2|E|mn}}^t,\\
\Varb{\indic{p}{t}}&=\Eb{\indic{p}{t}}-\Eb{\indic{p}{t}}^2.
\end{aligned}
\end{equation*}
Plugging in $t=t_{n,\lambda}=\frac{|E|nm}{n+m}\zar{\log n- \log \lambda }$, with $0<\lambda=\lambda_n \ll n$, we have
$$
\zar{1-\frac{m+n}{2|E|mn}}^t = \frac{\sqrt{\lambda}}{\sqrt{n}} (1+o(1)),
$$
hence
\begin{equation}\label{expeq}
\Eb{\indic{p}{t}}=\frac{m}{m+n}-\frac{\sqrt{\lambda} \, m (1+o(1))}{\sqrt{n}(m+n)},
\end{equation}
and
\begin{equation}\label{vareq}
\Varb{\indic{p}{t}} \leq \frac{m}{m+n}-\frac{m^2}{(m+n)^2} = \frac{mn}{(m+n)^2},
\end{equation}
where the last inequality is due to the monotonicity of $x(1-x)$ on $x\in (0,1/2)$.


Regarding the covariances in (\ref{vareqmain}), we have
\begin{equation*}
\begin{aligned}
\mathrm{Cov}\zar{\indic{p}{t},\indic{q}{t}}&=\PB{p,q\in\K_2\:\text{at time}\:t\Bigm|p,q\:\text{started in}\:\K_1}
-\PB{p\in\K_2\:\text{at time}\:t\Bigm| p\:\text{started in}\:\K_1}^2.
\end{aligned}
\end{equation*}
The first probability is independent of $p$ and $q$ (for $p\neq q$). In order to calculate it, we need the Markov chain of \emph{pairs} of particles, with three states: both particles are in $\K_1$; they are in different cliques; both are in $\K_2$. This chain is described by the following transition matrix:
$$ M=\begin{pmatrix} 1-\frac{1}{|E|n}&\frac{1}{|E|n}&0\\
					 \frac{n-1}{2|E|mn}&1-\frac{m+n-2}{2|E|mn}&\frac{m-1}{2|E|mn}\\
					 0&\frac{1}{|E|m}&1-\frac{1}{|E|m}
					 \end{pmatrix}.$$
This matrix has eigenvalues $1$, $1 - \frac{m + n-1}{|E| m n}$, and $1-\frac{m + n}{2 |E| m n}$, with right eigenvectors $(1,\, 1,\, 1)$, $\zar{\frac{m(m-1)}{n(n-1)}, \, -\frac{m-1}{n}, 1}$, and $\zar{-\frac{m}{n}, \, \frac{n - m}{2 n}, \, 1}$, which form an orthonormal basis w.r.t.~the stationary distribution $\zar{{n\choose 2}, \, nm, \, {m\choose 2}}/{n+m \choose 2}$. With the same method as above, we get the following:
\begin{equation}\label{e.covf}
\begin{aligned}
\mathrm{Cov}\zar{\indic{p}{t},\indic{q}{t}}&= \frac{m(m-1)}{(m+n)(m+n-1)(m+n-2)}\Bigg( m+n-2+\\
&\qquad+(m+n)\zar{1-\frac{m+n-1}{|E|mn}}^t
- 2(m+n-1)\zar{1-\frac{m+n}{2|E|mn}}^t\Bigg)\\
&\quad-\zar{\frac{m}{m+n}-\frac{m}{m+n}\zar{1-\frac{m+n}{2|E|mn}}^t}^2.
\end{aligned}
\end{equation}

We will plug in $t=t_{n,\lambda}=\frac{|E|mn}{n+m}\zar{\log n -\log\lambda}$ again. 
To keep track of lower order terms, we will use the following lemma:

\begin{lemma}
If $0\leq a,b$ with $a+b\leq 1$, and $t\ge 1$, then $(a+b)^t \leq a^t + tb$.
\end{lemma}

\proof For positive integer values of $t$, the claim has a simple probabilistic meaning. The LHS is the probability that at least one of two disjoint events (with probabilities $a$ and $b$) occurs all along $t$ independent tries. The RHS is an upper bound on the  probability that the first event happens always or the second event happens at least once. Since we do not see how to extend this argument for non-integer values of $t$, here is an analytic proof.

Given $b$, we first check the claim at the extremes of $a$, namely, $a=0$ and $a=1-b$. In both cases, the claim is obvious at $b=0$ and $b=1$, while the derivative in $b$ of the difference between the two sides has a fixed sign for $b\in (0,1)$, hence the claim also holds for these intermediate values of $b$. Next, we check the statement for the intermediate values $a\in (0,1-b)$. Now the derivative in $a$ of the difference between the two sides has a fixed sign, hence the claim follows.
\qqed

Now, the Taylor expansion $\exp(-\eps)=1-\eps+O(\eps^2)$ and the previous lemma imply that
$$
(1-\eps)^t = \exp (-\eps t) + O\zar{t \eps^2},
$$
as $\eps\to 0$ and $t\to\infty$. This gives
\begin{equation*}
\begin{aligned}
\zar{1-\frac{n+m}{2|E|mn}}^{t_{n,\lambda}}
&=\frac{\sqrt{\lambda}}{\sqrt{n}}+O\zar{\frac{\log n}{n^{2}m}} = \frac{\sqrt{\lambda}}{\sqrt{n}}+O\zar{\frac{\log n}{n^{5/2}}},\\
\zar{1-\frac{n+m-1}{|E|mn}}^{t_{n,\lambda}}
&=\frac{\lambda}{n}+O\zar{\frac{\log n}{n^{2}m}} = \frac{\lambda}{n}+O\zar{\frac{\log n}{n^{5/2}}},
\end{aligned}
\end{equation*}
using that $\Omega(\sqrt{n}) \leq m$. We now plug these into (\ref{e.covf}) to get
\begin{equation*}
\begin{aligned}
\mathrm{Cov}\zar{\indic{p}{t},\indic{q}{t}}
&=\frac{m(m-1)}{(m+n)(m+n-1)}\Bigg(1+\zar{1+\frac{2}{m+n-2} }\frac{\lambda}{n}
- 2\zar{1+\frac{1}{m+n-2} }\frac{\sqrt{\lambda}}{\sqrt{n}}+O\zar{\frac{\log n}{n^{5/2}}}\Bigg)\\
&\qquad-\frac{m^2}{(m+n)^2}\zar{1-\frac{\sqrt{\lambda}}{\sqrt{n}}+O\zar{\frac{\log n}{n^{5/2}}} }^2\\
&=\frac{m}{(m+n)^2}\Bigg\{ (m-1)\zar{1+\frac{1}{m+n-1} }\zar{1-2\frac{\sqrt{\lambda}}{\sqrt{n}}+ \frac{\lambda}{n} + O\zar{\frac{\sqrt{\lambda}}{n^{3/2}}} }\\
&\qquad\qquad\qquad
-m\zar{ 1-2\frac{\sqrt{\lambda}}{\sqrt{n}}+ \frac{\lambda}{n} + O\zar{\frac{\log n}{n^{5/2}}}}  \Bigg\}\\
&=\frac{m}{(m+n)^2}\Bigg\{-1+O\zar{ \frac{\sqrt{\lambda}}{\sqrt{n}} } \Bigg\}.
\end{aligned}
\end{equation*}
This and~(\ref{vareq}) together give
\begin{equation}\label{vareq2}
\begin{aligned}
\Varb{L^t} &= n(n-1) \, \mathrm{Cov}\zar{\indic{p}{t},\indic{q}{t}} + n \, \Varb{\indic{p}{t},\indic{q}{t}}\\
&\leq \frac{n m}{(m+n)^2} \Bigg\{(n-1)\zar{-1+O\zar{ \frac{\sqrt{\lambda}}{\sqrt{n}} } }+ n \Bigg\}\\
 &= \frac{n m}{(m+n)^2} O\zar{\sqrt{\lambda}\sqrt{n}}.
\end{aligned}
\end{equation}
On the other hand, (\ref{expeq}) gives us
\begin{equation}\label{expeq2}
\Eb{L^t} - \Eb{L^\infty} =  \frac{\sqrt{\lambda} \sqrt{n} \, m}{n+m} (1+o(1)).
\end{equation}

The key point is that the difference (\ref{expeq2}) is of larger order than the standard deviation of $L^t$ given by~(\ref{vareq2}), and also than the standard deviation of $L^\infty$ given by~(\ref{e.HypGeom}), if 
\begin{equation*}
\lambda=\lambda_n \gg \frac{n}{m^2},
\end{equation*}
which is satisfied for any $\lambda_n\to \infty$, since $m\ge c\sqrt{n}$. Thus, the difference between $L^t$ and $L^\infty$ should be possible to detect with high probability. 

More precisely, fix any sequence $\lambda_n\to \infty$ such that $\log\lambda_n \ll \log n$, so that we get a good lower bound for the cutoff. Then, (\ref{expeq2}) and Chebyshev's inequality with~(\ref{vareq2}) yield
\begin{equation}\label{e.LtCheb}
\begin{aligned}
\PB{L^t \geq \frac{mn}{m+n}-\frac{\sqrt{\lambda}}{2} \frac{m\sqrt{n}}{m+n}} &= 
\PB{L^t - \Es{L^t} \geq \frac{\sqrt{\lambda}}{2+o(1)}\frac{m\sqrt{n}}{m+n}}\\
&\leq\frac{\Varb{L^t} (4+o(1)) (m+n)^2}{  \lambda m^2 n} = O\zar{\frac{\sqrt{n}}{\sqrt{\lambda} m}},
\end{aligned}
\end{equation}
which goes to $0$ because $m=\Omega\zar{\sqrt{n}}$ and $\lambda=\lambda_n\to\infty$. Furthermore, Chebyshev's inequality with~(\ref{e.HypGeom}) yields
\begin{equation}\label{e.LinftyCheb}
\begin{aligned}
\PB{L^\infty \leq \frac{mn}{m+n}-\frac{\sqrt{\lambda}}{2} \frac{m\sqrt{n}}{m+n}} &= 
\PB{L^\infty - \Es{L^\infty} \leq -\frac{\sqrt{\lambda}}{2}\frac{m\sqrt{n}}{m+n}}\\
&\leq\frac{\Varb{L^\infty} 4 (m+n)^2}{  \lambda m^2 n} = O\zar{\frac{1}{\lambda}},
\end{aligned}
\end{equation}
which goes to 0 again. Comparing~(\ref{e.LtCheb}) and~(\ref{e.LinftyCheb}) shows that $L^t$ and $L^\infty$ are asymptotically singular as $n\to\infty$. This finishes the proof of the lower bound in~(\ref{e.largem}).

\subsection{The upper bound in the case $c\sqrt{n}\leq m \leq n$}\label{ss.large2}

For the upper bound in~(\ref{e.largem}), we are going to use Proposition~\ref{quickmix}, which says it is sufficient to prove mixing for the Bernoulli--Laplace diffusion in order to see mixing for the interchange process on $\Gnm$. We are going to present a coupling argument for the upper bound on the mixing time of the Bernoulli--Laplace model.

We define the coupling on the number of starting particles in $\K_1$ that reside in $\K_1$ at time $t$, for two such configurations. (Here we note that this number is between $n-m$ and $n$, since the $n$ particles of $\K_1$ do not fit into $\K_2$ if $m<n$).

Let us assume that the number of original particles in $\K_1$ is $k$. Then we have:
\begin{equation}\label{e.jump}
\begin{aligned}
&\PB{\text{jump to}\: k+1}=\frac{(n-k)^2}{2|E|nm},\qquad\qquad \PB{\text{jump to}\: k-1}=\frac{k(m-n+k)}{2|E|nm},\\
&\PB{\text{we remain at}\: k}=1-\frac{(n-k)^2}{2|E|nm}-\frac{k(m-n+k)}{2|E|nm}.
\end{aligned}
\end{equation}
Let the coupled chains be $X_t$ and $Y_t$. We couple them in the following way. Assuming $X_t=x_t$, $Y_t=y_t$, toss a fair coin to decide whether to attempt to move $X_t$. If it is heads, let $X_{t+1}$ be given by~(\ref{e.jump}), with $k=x_t$. If it is tails, move $Y_t$ with the analogous probabilities (simply replacing $x_t$ by $y_t$). Assuming $X_0=x$, $Y_0=y$, with $x\geq y$, we define $D_t=X_t-Y_t$, and then, for the jump probabilities of $D_t$, we have
\begin{equation}\label{e.Djump}
\begin{aligned}
\PB{D_{t+1}-D_t=1} &= \frac{(n-x_t)^2}{2|E|nm}+\frac{y_t(m-n+y_t)}{2|E|nm} \\
\PB{D_{t+1}-D_t= -1} &= \frac{(n-y_t)^2}{2|E|nm}+\frac{x_t(m-n+x_t)}{2|E|nm} \\
\PB{D_{t+1}-D_t=0} &= 1 - \frac{(n-x_t)^2}{2|E|nm} - \frac{y_t(m-n+y_t)}{2|E|nm} - \frac{(n-y_t)^2}{2|E|nm}-\frac{x_t(m-n+x_t)}{2|E|nm}.
\end{aligned}
\end{equation}
Thus, for the expectation of the jump:
\begin{equation}\label{e.drift}
\begin{aligned}
\Eb{D_{t+1}-D_t \md X_t=x_t,Y_t=y_t}=-\frac{1}{2|E|nm}(x_t-y_t)(n+m)=\frac{-D_t(n+m)}{2|E|nm}.
\end{aligned}
\end{equation}
Iterating this, starting with $0 \leq x - y \leq m$,
\begin{equation}\label{couplingexp}
\begin{aligned}
\Ess{x,y}{D_t}\leq\zar{1-\frac{n+m}{2|E|nm}}^tm\sim m\exp\zar{-t\frac{n+m}{2|E|nm}}.
\end{aligned}
\end{equation}
This expectation gets close to 0 only for some $t$ that is not good enough for the bound~(\ref{e.largem}) that we are aiming it. It will nevertheless be useful in the forthcoming argument, which we designed after the treatment of the Ehrenfest urn model in \cite[Theorem 18.3]{LPW}.

The process $D_t$ is somewhat similar to a random walk on the integers, except that it has a drift and laziness that depends not only on the current location $D_t$, but even on the states $\left(X_t,Y_t\right)$. To simplify this situation, we will couple $D_t$ to a ``symmetrized'' process $S_t$, still driven by the events of $\zar{X_t,Y_t}$, and this $S_t$ to a slower ``copycat process'' $L_t$, which moves the same way but with a fixed (maximal) laziness. 

Given $X_t=x$ and $Y_t=y$, the symmetrized process is defined from~(\ref{e.Djump}) by
\begin{equation*}
\begin{aligned}
\Pb{S_{t+1}=S_t+1}=\Pb{S_{t+1}=S_t-1}&:=\frac{\Ps{D_{t+1}=D_t+1}+\Ps{D_{t+1}=D_t-1}}{2}=:p_{x,y}\,,\\
\Pb{S_{t+1}=0}&:=1 - 2\,p_{x,y}\,.
\end{aligned}
\end{equation*}
Note that the negative drift in~(\ref{e.drift}) shows that for $D_t$ the probability of going left (in the negative direction) is always larger than for $S_t$. Thus we can couple $D_t$ and $S_t$ as follows. If $D_t$ goes left, let $S_t$ go left with probability $\frac{\Ps{S_{t+1}=S_t - 1}}{\Ps{D_{t+1}=D_t-1}}$, right with probability $1-\frac{\Ps{S_{t+1}=S_t - 1}}{\Ps{D_{t+1}=D_t-1}}$. If $D_t$ goes right, let $S_t$ also go right. Hence, the marginal distributions correspond to the original $D_t$ and $S_t$ and we have $D_t\leq S_t$ (if $D_0\leq S_0$). Also note that $S_t$ has the same laziness as $D_t$, given $(X_t,Y_t)$.

Now let the lazy copycat process be the time-homogeneous random walk given by
\begin{equation*}
\begin{aligned}
\Pb{L_{t+1}=L_t\pm 1}=\frac{m}{2|E|n},\qquad\qquad \Pb{L_{t+1}=L_t}=1-\frac{m}{|E|n}.
\end{aligned}
\end{equation*}
Note that the maximal laziness of $S_t$, achieved at $(x,y)=(n,n-m)$, is indeed the laziness of $L_t$ given here. Now the coupling between $S_t$ and $L_t$ is as follows. Let the sequence of non-lazy moves made by $S_t$ be $s_1,s_2,\dots \in \{\pm 1\}$. Let us now assume that $X_t=x$ and $Y_t=y$. Then $S_t$ moves left (or right) with probability $p_{x,y}$. If $S_t$ does move, let $L_t$ move with probability $\frac{m}{|E|n}\cdot\frac{1}{p_{x,y}}$. The direction of the move of $L_t$ is the first  move from the list $\{s_1,s_2,\dots\}$ that have not been used yet for the copycat process (this can be done, since by the coupling there are at least as many moves of $S_t$ as there are of $L_t$). This way, $L_t$ ``moves like a shadow of $S_t$'', just with a smaller speed.

So, if $\tau^D$, $\tau^S$, $\tau^L$ are the times at which $D_t$, $S_t$, $L_t$ reach 0, respectively, then we have 
\begin{equation}\label{e.DSL}
\Pb{\tau^D>u}\leq \Pb{\tau^S>u} \leq \Pb{\tau^L>u},\qquad \text{for all }u>0.
\end{equation}
We will need the following statement:

\begin{propo}
Let $L_t$ be a symmetric random walk with laziness probability $1-\frac{1}{M}$. Let $\tau^L$ be the time when $L_t$ reaches 0. Then there exists $c_1\in\mathbb{R}$ and $u_0\in\mathbb{Z}^+$, such that, for all $u>u_0$,
$$
\Pss{k}{\tau^L > u\, M} \leq \frac{c_1k}{\sqrt{u}}.
$$
\end{propo}

\proof  Let $N_t$ be a simple symmetric random walk. By Theorem 2.26 in \cite{LPW}, we know that if $\tau^N$ is the time it takes $N_t$ to reach 0, then
\begin{equation}\label{e.sqrtu}
\Pss{k}{\tau^N > u} \leq \frac{ck}{\sqrt{u}}.
\end{equation}
We can couple $L_t$ and $N_t$ such that $L_t$ is the lazy copycat version of $N_t$. Then, if $\tau_u$ is the (almost surely finite) time it takes for $L_t$ to move $u$ times, then 
\begin{equation}\label{e.tauNL}
\Pss{k}{\tau^N>u}=\Pss{k}{\tau^L>\tau_u}.
\end{equation}
For the right hand side,
\begin{equation}\label{tauineq}
\begin{aligned}
\Pss{k}{\tau^L>\tau_u}&=\sum_{t=1}^{\infty}\Pss{k}{\tau^L>t}\Pb{\tau_u=t}\geq \sum_{t=\frac{Mu}{2}}^{2Mu}\Pss{k}{\tau^L>k}\Pb{\tau_u=t}\\
&\geq\:\Pss{k}{\tau^L>2Mu} \, \PB{\frac{Mu}{2}\leq\tau_u\leq 2Mu}.
\end{aligned}
\end{equation}
Since $\tau_u=\xi_1+\ldots+\xi_u$, where $\xi_i\sim\mathsf{Geom}\zar{\frac{1}{M}}$, the weak law of large numbers tells us that
$$
\PB{\Bigl\lvert\frac{\xi_1+\ldots+\xi_u}{u}-M\Bigr\rvert>\varepsilon}=\PB{\left\lvert\tau_u-Mu\right\rvert>\varepsilon Mu}\rightarrow 0,\qquad\text{as }u\rightarrow\infty.
$$
Thus, for any $\delta>0$ and $u>u_0$ large enough, we have
$\PB{\frac{Mu}{2}\leq\tau_u\leq 2Mu}>1-\delta$. Plugging this into~\eref{tauineq} and using~\eref{e.sqrtu}, identity~\eref{e.tauNL} gives us 
$$
\Pss{k}{\tau^L>2Mu}\leq \frac{ck}{\sqrt{u}(1-\delta)}.
$$
Taking $c_1=\frac{c}{1-\delta}$, we are done.
\qqed

Now, applying this proposition to our copycat process $L_t$ with $M=\frac{|E|n}{m}$, from~\eref{e.DSL} we get
\begin{equation*}
\begin{aligned}
\Pss{x,y}{\tau^D>s+u|E|\frac{n}{m} \Bgiv D_s}=\Pss{D_s}{\tau^D>u|E|\frac{n}{m}}\leq\Pss{D_s}{\tau^L>u|E|\frac{n}{m}}\leq\frac{c_1D_s}{\sqrt{u}}.
\end{aligned}
\end{equation*}
Then, taking expectation over $D_s$, using~\eref{couplingexp}:
\begin{equation}
\begin{aligned}
\Pss{x,y}{\tau>s+u|E|\frac{n}{m}}\leq\frac{c_1m\exp\zar{-s\frac{n+m}{2|E|nm}}}{\sqrt{u}}.
\end{aligned}
\end{equation}
Hence, we can choose $s=\frac{|E|nm}{n+m}\log n$ and $u=\alpha\frac{m^2}{n}$ to get
\begin{equation*}
\begin{aligned}
\Pss{x,y}{\tau>s+u}\leq \frac{c_1}{\sqrt{\alpha}}.
\end{aligned}
\end{equation*}
This means that, for any $\eps>0$,
\begin{equation}
\begin{aligned}
t_{\mathrm{mix}}(\eps) \leq \frac{|E|nm}{n+m}\log n + O(|E|m)\,, 
\end{aligned}
\end{equation}
end the proof of (\ref{e.largem}) is complete.

\subsection{The case $1\ll m\ll \sqrt{n}$}\label{ss.medium}

We will look at the event that every particle that started in $\K_2$ (let us call these \emph{red} particles) is in $\K_1$ at some time $t$. Note that the probability in the stationary distribution $\pi$ is asymptotic to $\exp\zar{\frac{2m^2}{n}}=1-o(1)$, because $m=o\zar{\sqrt{n}}$.
\medskip

For a lower bound on the mixing time, take $t=2|E|m\log m-\lambda |E| m$. The probability of every red particle being in $\K_1$ at time $t$ can be upper bounded by the probability that every red particle has \emph{at some point} visited $\K_1$. The time $\tau$ it takes for this to happen can be described as follows:
\begin{equation*}
\begin{aligned}
\tau=\sum_{k=0}^{m-1}\tau_k \qquad \text{where} \quad \tau_k\sim \mathrm{Geom}\zar{\frac{m-k}{2|E|m}} \: \text{are independent.}
\end{aligned}
\end{equation*}
Hence we have $\Es{\tau}= (2+o(1))|E|m\log m$, and $\Vars{\tau}\leq 4|E|^2m^2$, and so, by Chebyshev's inequality,
\begin{equation*}
\begin{aligned}
\Pb{\tau < 2|E|m\log m-\lambda |E|m}\leq \frac{4|E|^2m^2}{\lambda^2|E|^2m^2}=\frac{4}{\lambda^2}.
\end{aligned}
\end{equation*}
Hence for any $\delta>0$, if $\lambda>0$ is large enough, then at $t=2|E|m\log m-\lambda|E|m$ we have
\begin{equation*}
\begin{aligned}
\max_{\sigma \in S_{n+m}}\norm{P^t\zar{\sigma,\cdot}-\pi}\geq 1-\delta\,,
\end{aligned}
\end{equation*}
and thus the lower bound of~(\ref{e.smallm}) follows.
\medskip

For the upper bound, let us look at the probability that a red particle $p$ resides in $\K_2$ at time $t$:
$$P^t\zar{p\in\K_2}=\frac{m}{m+n}+(1+o(1))\frac{n}{m+n}\exp\zar{-t\frac{m+n}{2|E|mn}}.$$
Plugging in $t=2|E|m\log m+\lambda|E|m$, we get:
$$P^t\zar{p\in\K_2}=\frac{m}{m+n}+(1+o(1))\frac{n}{m+n}\frac{1}{m^{1+\frac{m}{n}}}\exp\zar{-\frac{\lambda}{2}\zar{1+\frac{m}{n}}}.$$
Since $m=o\zar{\sqrt{n}}$, we can take a union bound for the probability that \emph{any} red particle is in $\K_2$:
\begin{equation*}
\begin{aligned}
P^t\zar{\text{reds}\:\cap \K_2 \neq \emptyset}\leq& \frac{m^2}{m+n}+\frac{mn}{m+n}\frac{1}{m^{1+\frac{m}{n}}}\exp\zar{-\frac{\lambda}{2}\zar{1+\frac{m}{n}}}+o(1)\\
\leq &\exp\zar{-\frac{\lambda}{2}}+o(1).
\end{aligned}
\end{equation*}
That is, the contribution to the total variation distance of $P^t(\sigma,\cdot)$ and the stationary distribution $\pi$ from the part of the probability space where any red particle still resides in $\K_2$ is at most $\exp\zar{-\frac{\lambda}{2}}+o(1)$. On the other hand, on the event that all red particles are in $\K_1$, the total variation distance is small due to the permutations in $\K_1$ having mixed in $O\zar{n\log n}$ time, as explained in Subsection~\ref{ss.within}. Hence, for large $\lambda$, the total variation distance at $t=2|E|m\log m + \lambda|E|m$ is small, and the upper bound of~(\ref{e.smallm}) follows.

\subsection{The case $m\asymp 1$}\label{ss.small}

We are going to prove that the interchange process $X'_t$ over the ``half-symmetrized'' graph $G'_{n,m}$, introduced right before Proposition~\ref{p.coupling2}, when $m$ is fixed, satisfies the total variation distance bounds of~(\ref{e.constm}). By Proposition~\ref{p.coupling2}, this is inherited to the interchange process over $G_{n,m}$, and hence part (iii) of Theorem~\ref{maintheo} will be proved.

%
%
%

As in the previous subsection, we will look at the event $\O$ that all the $m$ red particles starting in $\K_2$ are in $\K_1$. Note that $\O$ has stationary probability $1-o(1)$ as $n\to\infty$. Furthermore, conditioned on $X'_t \in \O$, the red particles are uniformly located in $\K_1$, and the identity of the $m$ non-red particles in $\K_2$ is also uniform, hence 
$$\Ps{X'_t \not\in \O}-o(1)  \leq d_{\mathrm{TV}}(X'_t,\pi) \leq \frac{\Ps{X'_t \not\in \O}}{1-o(1)}\,.$$
So, it is enough to bound the probability $\Ps{X'_t \not\in \O}$. Now the chain is not as symmetric as before, hence exact calculations are not viable, but we still can give good enough bounds.

For each particle $i\in\{n+1,\dots,n+m\}$ started in $\K_2$, let $\tau_i$ be the first time when it enters $\K_1$, and let $\tau:=\max\{\tau_i : n+1 \leq i \leq n+m\}$.  Clearly, $\tau_i$ stochastically dominates a geometric random variable with success probability $\frac{1}{|E_n|}$, since even if the particle is at the bridge vertex $n+1$, we need a bridge transposition to occur. This implies that $\Ps{X'_t\in\O} \leq \Ps{\tau \leq t} < 1-\eps$ holds for all $t<Kn^2$, where $K$ is large if $\eps>0$ is small. This gives the lower bound in~(\ref{e.constm}).

For an upper bound, consider the Markov chain with 3 states on Figure~\ref{f.M3}, a projection of the movement of a single particle.

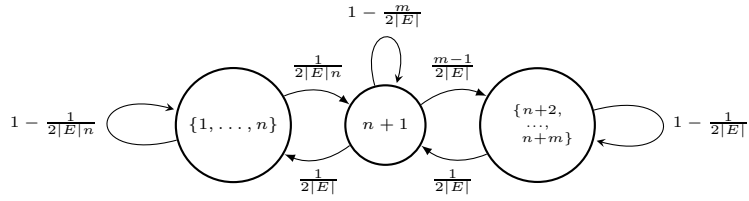
\begin{figure}[h!]
\centering
\begin{tikzpicture}
\node[state] (1) {$\{1,\dots,n\}$};
\node[state] (2) [right of=1] {$n+1$};
\node[state] (3) [right of=2] {$\begin{smallmatrix}\{n+2,\\ \dots,\\ \quad n+m\}\end{smallmatrix}$};
\draw[-latex] (1) to[bend left=30] node[above] {$\frac{1}{2|E|n}$} (2);
\draw[-latex] (2) to[bend left=30] node[below] {$\frac{1}{2|E|}$} (1);
\draw[-latex] (2) to[bend left=30] node[above] {$\frac{m-1}{2|E|}$} (3);
\draw[-latex] (3) to[bend left=30] node[below] {$\frac{1}{2|E|}$} (2);

\path[]
    (1) edge [loop left] node {$1-\frac{1}{2|E|n}$} (1)
    (2) edge [loop above] node {$1-\frac{m}{2|E|}$} (2)
    (3) edge [loop right] node {$1-\frac{1}{2|E|}$} (3);
\end{tikzpicture}
\caption{A projected chain for a single particle in the interchange process over $G'_{n,m}$.}\label{f.M3}
\end{figure}

A standard calculation gives that $\E\tau_i \asymp n^2$ both for $i=n+1$ and for $i\in\{n+2,\dots,n+m\}$. This implies, by Markov's inequality and a union bound, that $\Ps{\tau > Kn^2} < \eps$, if $K$ is large enough. Moreover, the additional time $\gamma_i$ after $\tau_i$ when particle $i$ is first back at $\K_2$ follows a geometric random variable $\gamma_i$ with success probability $\frac{1}{2n|E_n|}$, and hence, for any $t\asymp n^2$, we have $\Ps{t < \min_i \gamma_i} > 1-o(1)$. Altogether, for $t=Kn^2$, with $K$ large enough, we have  
$$\Ps{X'_t \in \O} \geq \Ps{\tau_i < t <\tau_i+\gamma_i \textrm{ for all }i} > 1-2\eps\,,$$
which gives the upper bound in~(\ref{e.constm}), and finishes the proof of Theorem~\ref{maintheo}.
\qqed
%
%

\section{The labelled exclusion process in the complete graph}\label{s.exclu}

First of all, here are the results of Lacoin and Leblond \cite{lale} on the labelled exclusion process, written here for the lazy version:

\begin{theo}[\cite{lale}]\label{t.laleold}
For the $\frac12$-lazy exclusion process on the complete graph with $n$ vertices and $k$ labelled particles, for every $\eps\in (0,1)$ there exists $\beta > 0$ such that, for every $k$ and $n$,
$$
t_\mix(1-\eps) \geq n\log k-\beta n\,.
$$
Moreover, if $\lim_{n\to\infty} k(n)/\sqrt{n}= 0$, then for every  for every $\eps\in (0,1)$ there exists $\beta > 0$ such that, for every $k$ and $n$,
$$
t_\mix(\eps) \leq n\log k+\beta n\,.
$$
\end{theo}

\proofof{Theorem~\ref{t.lalenew}} It remains to prove that, for any $c>0$, $\eps>0$, $\delta\in(0,1)$, if $n$ is large enough and $k\geq c\sqrt{n}$, then, at time $T:=(1+\eps)\, n \log k$ we are at TV-distance at most $\delta$ from stationarity. 

Let $\K$ be the subset of vertices where the labelled particles start, let $L_T$ be the number of labelled particles at time $T$ who have never left $\K$, and color them purple. Clearly, as $n\to\infty$,
$$
\Es{L_T} = k \zar{1-\frac{n-k}{2{n \choose 2}}}^T \sim k \exp\zar{-(1+\eps) \frac{n-k}{n} \log k} = \exp\zar{ \zar{\frac{k}{n}-\eps\frac{n-k}{n}} \log k }.
$$

We will first assume that $k \leq n/2$. Then, the above formula for $\Es{L_T}$ and Markov's inequality give
\begin{equation}\label{e.leftover}
 \P\zar{L_T \leq   \exp\zar{ \zar{\frac{k}{n}-\frac\eps2\frac{n-k}{n}} \log k }  } \geq 1-k^{-\eps/4}\,.
\end{equation}
The point is that this is $o(\sqrt{k})$ with high probability, hence it will be possible to use the upper bound of Theorem~\ref{t.laleold} for the location of these leftover particles within the clique $\K$. Of course, this upper bound also  follows from our proof in Subsection~\ref{ss.medium}, since at this time all particles have left $\K$ whp, and they are at uniform random positions in $\K^c$. 

The non-purple particles may be either in $\K^c$ or in $\K$, but in either case, since moving between the two parts always happens to a uniform random location, their positions within their parts is uniform. Moreover, the attempted transpositions of the non-purple particles within themselves and with the empty locations are independent of the attempted transpositions of purple particles within themselves and with the empty locations, except for their numbers. Thus, using the argument of (\ref{e.product1},\ref{e.product2}), once enough time has passed so that 
\begin{enumerate}
\item[(1)] the number of particles that are currently in $\K$ is close to its stationary distribution,
\item[(2)] and the location of the purple particles (conditioned on their number) is close to uniform within $\K$, 
\end{enumerate}
then the entire configuration is close to stationarity.

The process of item (1) is simply a time-changed Bernoulli--Laplace model, whose mixing time can be estimated by our previous results. Namely, we have now two urns, of sizes $k$ and $n-k$, with $\Omega(\sqrt{n-k}) \leq k \leq n-k$, hence case~(\ref{e.largem}) of Theorem~\ref{maintheo} applies, except that the speed of the bridge transpositions is not $\frac{1}{2|E|}$, but $\frac{k(n-k)}{ 2{n \choose 2}}$. Of course, we have a random time change, but, by the law of large numbers, the mixing times can just be multiplied by these speeds. So, we get a mixing time
$$
\frac{1}{2|E|} \frac{2{n\choose 2}}{k(n-k)} \frac{|E| k (n-k)}{k+n-k} \log (n-k) = \frac{1+o(1)}{2} n \log n \leq (1+o(1))\, n \log k \,,
$$
with cutoff, where the last inequality used that $c\sqrt{n}\leq k$.

By \cite[Proposition 4.7]{LPW}, there is an optimal coupling between our time-changed  Bernoulli--Laplace process at time $T$ and its stationary distribution, such that the number of particles currently in $\K$ is the same in the two, with probability close to 1. This coupling can be pulled back to a coupling between the exclusion process and its stationary distribution. Conditioning on the event of successful coupling can change the probability of any event only by a small additive amount, hence the bound of~(\ref{e.leftover}) on $L_T$ still holds with high probability. 

To understand item (2), condition on $L_T=\ell$ and on the identity of these $\ell$ particles, for any fixed $\ell$ that satisfies the bound of~(\ref{e.leftover}). Look at the movement of the $\ell$ purple particles within $\K$ during the $T$ steps. We have the conditioning that, for each purple particle, there is no transposition going to $\K^c$ before time $T$, while, for every non-purple particle, there is at least one transposition going to $\K^c$. This conditioning certainly changes the distribution of the number $\gamma$ of attempted transpositions by time $T$ within $\K$, but it  does not change the fact that the number $\gamma^*$ among these transpositions that actually happen still has distribution $\mathsf{Binom}(\gamma,1/2)$, and it does not break the symmetry between these transpositions: each is uniformly distributed among the edges in $\K$, independently from each other. Therefore, we only need to determine if, under the conditioning, $\gamma$ is large enough with high probability for the mixing of the purple particles in $\K$.

Without the conditioning, the distribution of $\gamma$ is $\mathsf{Binom}\zar{T,\frac{{k\choose 2}}{{n\choose 2}}}$. This has expectation $(1+\eps+o(1)) \frac{k^2 \log k}{n}$, which goes to infinity with $n$ because of the condition $k\geq c\sqrt{n}$. Hence, by a standard large deviations bound, 
\begin{equation}\label{e.gamma1}
\PB{\gamma < \zar{1+\frac{\eps}{2}} \frac{k^2 \log k}{n} } < \exp\zar{-c_\eps  \frac{k^2 \log k}{n} }\,,
\end{equation}
with some $c_\eps > 0$ that depends only on $\eps$. Using the bound of~(\ref{e.leftover}) on $\ell$,
\begin{equation}\label{e.gamma2}
\zar{1+\frac{\eps}{2}} \frac{k^2 \log k}{n} \geq \zar{1+\frac{\eps}{2}} k \log \ell\,,
\end{equation}
and we also have $\ell \ll \sqrt{k}$, hence the upper bound in Theorem~\ref{t.laleold} would tell us that, without the conditioning, $\gamma$ would be large enough. But what is the effect of the conditioning?
 
Let $\alpha_i$ and $\alpha^*_i$, for $i=1,\dots,\ell$, denote the number of attempted and actual transpositions between the $i^\mathrm{th}$ purple particle and $\K^c$. Similarly, let $\beta_j$ and $\beta^*_j$, for $j=1,\dots,k-\ell$, denote the number of attempted and actual transpositions between the $j^\mathrm{th}$ non-purple particle and $\K^c$. We want to show that
$$
\PB{ \gamma < \zar{1+\frac{\eps}{2}} \frac{k^2 \log k}{n} \md  \forall i\, \alpha^*_i=0,\ \forall j\, \beta^*_j\geq 1} \to 0,
$$
as $n\to\infty$. The conditioning on $\{\forall i\, \alpha^*_i=0\}$ can only stochastically increase the distribution of $\gamma$, hence we can ignore it. On the other hand, the conditioning on $\{\forall j\, \beta^*_j\geq 1\}$ will turn out not to be too drastic, because the event itself is not extremely unlikely. For each $j$,
\begin{equation*}
\Pb{ \beta^*_j = 0} = \zar{ 1-\frac12\frac{n-k}{{n\choose 2}} }^T = \exp\zar{ -(1+\eps+o(1))\frac{n-k}{n}\log k}.
\end{equation*}
Inductively adding more and more $j$'s, one can easily see that the events $\{ \beta^*_j \geq 1\}$ are all negatively correlated with each other, hence
\begin{equation}\label{e.beta}
\begin{aligned}
\Pb{ \forall j\, \beta^*_j\geq 1} \geq \Pb{ \beta^*_j \geq 1}^{k-\ell} 
&\geq \zar{1-\exp\zar{ -(1+\eps+o(1))\frac{n-k}{n}\log k}}^k\\
&= \exp\zar{  -(1+\eps+o(1))\frac{n-k}{n}\log k + \log k } \\
&\geq \exp\zar{ (1+\eps)\frac{k}{n}\log k},
\end{aligned}
\end{equation}
where the last inequality holds if $n$ is large enough.

Now, denoting the events $\cG:=\left\{\gamma < \zar{1+\frac{\eps}{2}} \frac{k^2 \log k}{n}\right\}$ and $\cB:=\left\{\forall j\, \beta^*_j\geq 1\right\}$, the bounds~(\ref{e.gamma1}) and~(\ref{e.beta}) give us
$$
\Ps{\cG \md \cB} = \frac{\Ps{\cG \cap \cB}}{\Ps{\cB}} \leq  \frac{\Ps{\cG}}{\Ps{\cB}} < \exp\zar{ -c_\eps  \frac{k^2 \log k}{n}  + (1+\eps)\frac{k}{n}\log k  }\to 0,
$$
where the convergence to 0 holds because $k$ goes to infinity. As explained above, this finishes the proof for the case $k\leq n/2$.

Finally, we reduce the case of $k > n/2$ to the case of $k\leq n/2$ by the following simple trick. Color the first $n/2$ particles red, the remaining $k-n/2$ particles blue. Also, think of the $n-k$ unlabelled empty locations as labelled white particles. By time $T=(1+\eps) n\log n$, the red particles are $\delta$-close to stationarity, and the blue and white particles together are $\delta$-close to stationarity, as labelled exclusion processes. This means that the positions of the red, blue, white particles relative to each other, the permutation of the red particles among each other, and the permutation of the blue and white particles among each other, this data altogether is $2\delta$-close to stationarity. Moreover, the attempted transpositions that have happened within the red and within the blue-white groups are also independent from each other, except for their numbers. Again by the argument of (\ref{e.product1},\ref{e.product2}), this means that the entire configuration is close to stationarity, and we are done. (Note where the laziness for the $k=n-1,n$ cases is used: without the laziness, the sum of the numbers of actual transpositions between the differently coloured groups would be fixed at any given time, and although equations (\ref{e.product1},\ref{e.product2}) would still hold, the resulting conditional TV-distances would not be small at all: at any odd time, the measure would be concentrated on odd permutations.)
\qqed

\end{document}